\documentclass[11pt,letterpaper]{article}

\usepackage{fullpage}
\usepackage{amsmath,amsthm,amssymb}

\newtheorem{theorem}{Theorem}
\newtheorem{proposition}[theorem]{Proposition}
\newtheorem{lemma}[theorem]{Lemma}
\newtheorem{corollary}[theorem]{Corollary}

\newcommand{\si}{\sigma}
\newcommand{\ii}[1]{\textit{#1}}
\newcommand{\al}{\alpha}
\newcommand{\set}[1]{\{#1\}}
\newcommand{\sun}{{\rm SUN}}
\newcommand{\scn}{{\rm SCN}}
\newcommand{\pre}[1]{{#1_{\rm pre}}}
\newcommand{\floor}[1]{\left\lfloor{#1}\right\rfloor}

\allowdisplaybreaks[1]

\long\def\symbolfootnote[#1]#2{\begingroup%
\def\thefootnote{\fnsymbol{footnote}}\footnote[#1]{#2}\endgroup}

\begin{document}

\title{Constructions and nonexistence results for suitable sets of permutations}
\author{Justin H.C. Chan \and Jonathan Jedwab}
\date{March 23, 2016 (revised November 30, 2016)}
\maketitle

\symbolfootnote[0]{
Department of Mathematics, Simon Fraser University, 8888 University Drive, Burnaby BC V5A 1S6, Canada.
\par
J.~Jedwab is supported by NSERC.
\par
Email: {\tt jhc34@sfu.ca}, {\tt jed@sfu.ca}
\par
2010 Mathematics Subject Classification 05A05, 05B20
\par
}

\begin{abstract}
A set of $N$ permutations of $\set{1,2,\dots, v}$ is \ii{$(N,v,t)$-suitable} if each symbol precedes each subset of $t-1$ others in at least one permutation. The central problems are to determine the smallest $N$ for which such a set exists for given $v$ and $t$, and to determine the largest $v$ for which such a set exists for given $N$ and~$t$. These extremal problems were the subject of classical studies by Dushnik in 1950 and Spencer in 1971. 
We give examples of suitable sets of permutations for new parameter triples $(N,v,t)$.  We relate certain suitable sets of permutations with parameter $t$ to others with parameter $t+1$, thereby showing that one of the two infinite families recently presented by Colbourn can be constructed directly from the other.  We prove an exact nonexistence result for suitable sets of permutations using elementary combinatorial arguments. We then establish an asymptotic nonexistence result using Ramsey's theorem.
\end{abstract}



\section{Introduction}
\label{sec:intro}
A set of $N$ permutations of $[v] = \set{1,2,\dots,v}$ is \ii{$(N,v,t)$-suitable} if each symbol precedes each subset of $t-1$ others in at least one permutation; necessarily we must have $t \le \min(v,N)$. We represent such a set as an $N\times v$ array called an \ii{$(N,v,t)$-suitable array}. For example, $\set{2413,3421,1423}$ is a $(3,4,3)$-suitable set of permutations and its corresponding array is
$$
\begin{bmatrix}
2&4&1&3 \\
3&4&2&1 \\
1&4&2&3
\end{bmatrix} .
$$
Given an $(N,v,t)$-suitable array, we can readily form an $(N+1,v,t)$-suitable array by adding an arbitrary extra row, and an $(N,v-1,t)$-suitable array by removing all occurrences of a single symbol (and left-justifying the remaining symbols). This simple observation motivates two fundamental extremal problems:

\begin{enumerate}
\item[(P1)]
Given $v$ and $t$, what is the smallest $N$ for which an $(N,v,t)$-suitable array exists? 
We denote this as $N(v,t)$ (following~\cite{dushnik}), which is well-defined: the $v \times v$ array whose initial elements are $1, 2, \dots, v$ is $(v,v,t)$-suitable for each $t \le v$ and so $N(v,t) \le v$.

\item[(P2)]
Given $N$ and $t$, what is the largest $v$ for which an $(N,v,t)$-suitable array exists? 
We denote this as $\sun(t,N)$ (following~\cite{colbourn}). It is 
well-defined for $t\ge3$: we then have $\sun(t,N)\le 2^{2^N}$ \cite{spencer}, and $\sun(t,N) \ge N$ by reference to the $(v,v,t)$-suitable example just described.
But $\sun(2,N)$ is not well-defined for $N \ge 2$, because the $N \times v$ array whose first two rows are
$\begin{bmatrix} 1 & 2 & \dots & v-1 & v \end{bmatrix}$
and
$\begin{bmatrix} v & v-1 & \dots & 2 & 1 \end{bmatrix}$
is $(N,v,2)$-suitable for arbitrarily large~$v$.
\end{enumerate}

In 1950, Dushnik \cite{dushnik} introduced problem (P1), showing by combinatorial arguments that $N(v,t) = v-j+1$ for each $j$ satisfying $2\le j\le \sqrt{v}$ and for each~$t$ satisfying
$$
\floor{\frac{v}{j}} +j-1 \le t < \floor{\frac{v}{j-1}} +j-2.
$$
This determines $N(v,t)$ exactly for all $t$ in the range
$$
\floor{\frac{v}{\floor{\sqrt{v}}}} + \floor{\sqrt{v}} - 1 \le t < v.
$$
In particular, when the lower bound is attained (arising by taking $j = \floor{\sqrt{v}}$), both $v$ and $N(v,t)$ grow as~$\Theta(t^2)$.

Spencer \cite{spencer} continued the study of problem (P1) in 1971. 
Under the condition that $t \ge 3$ is fixed, he used a theorem due to Erd\H{o}s and Szekeres \cite{erdos-szekeres} to show that $N(v,t) \ge \log_2 \log_2 v$ (or equivalently $\sun(t,N)\le 2^{2^N}$), and Sperner's lemma \cite{sperner} and the Erd\H{o}s-Ko-Rado theorem \cite{erdos-ko-rado} to show that $N = O(\log_2 \log_2 v)$ as $v \to \infty$. 

F{\"u}redi and Kahn \cite{furedi-kahn} studied problem (P1) in 1986, using probabilistic methods to show that $N(v,t) \le t^2(1+\log(v/t))$ for all $t$ and~$N$.
Kierstead \cite{kierstead} refined this result when $t$ is approximately~$\log{v}$.

In a recent paper, Colbourn \cite{colbourn} studied problem~(P2) by linking suitable sets of permutations to a variety of combinatorial structures explicitly.  
He showed that $\sun(t,N) = \Theta(2^{2^N})$ for fixed $t \ge 3$, extending Spencer's result, by using a connection with binary covering arrays~\cite{lawrence}.
He examined the case when $v$ and $N$ both grow as $t^2$, by making a connection with Golomb rulers and their variants \cite{drakakis}, \cite{erdos-sidon}, \cite{erdos-turan}.
He found results for the case when $t$ is $O(\log N)$ by making a connection with Hadamard matrices~\cite{horadam} and Paley matrices~\cite{paley}.

Our principal interest in this paper is Colbourn's study of problem (P2) for values of $t$ that are intermediate between those described above.
His results (and ours) are conveniently described in terms of the quantity 
\begin{equation}
\scn(t,N): = \sun(t,N)-N,
\label{scn-defn}
\end{equation}
whose motivation will be described in Section~\ref{sec:cores}. 
Colbourn \cite[Section~1]{colbourn} derives the value of $\scn(2s,N)$ for all $N < s(s+1)$, and the value of $\scn(2s+1,N)$ for all $N < (s+1)^2$, and gives the following lower bounds for the next largest value of~$N$.
\begin{theorem}[Colbourn] 
\label{thm:colbourn}
\mbox{}
\begin{enumerate}
\item[(i)]  {\rm\cite[Lemma 1.3]{colbourn}}
$\scn(2s,s(s+1)) \ge s+2$ for all $s \ge 2$.

\item[(ii)]  {\rm\cite[Lemma 1.4]{colbourn}}
$\scn(2s+1,(s+1)^2) \ge s+2$ for all $s \ge 1$.
\end{enumerate}
\end{theorem}

Parts (i) and (ii) of Theorem~\ref{thm:colbourn} are proved in \cite{colbourn} by two explicit constructions of families of suitable arrays.
In Section~\ref{sec:tt+1} we will prove a general relation between certain suitable arrays with parameter $t$ and others with parameter~$t+1$. One consequence is that part (ii) of Theorem~\ref{thm:colbourn} can be obtained directly from part~(i).

Colbourn \cite{colbourn} states (without proof) that the inequality of Theorem~\ref{thm:colbourn}~(i) is actually an equality. We demonstrate by example in Section~\ref{sec:cores} the new results that $\scn(3,4) \ge 8$ and $\scn(5,9) \ge 5$, so that equality does not hold in Theorem~\ref{thm:colbourn}~(ii) for the cases $s=1$ and $s=2$. This appears to suggest that the inequality of Theorem~\ref{thm:colbourn}~(ii) is not sharp in general, but we shall show in Section~\ref{sec:oddcase} using elementary combinatorial arguments that this is not the case:

\begin{theorem}
\label{thm:odd_s}
$\scn(2s+1,(s+1)^2) = s+2$ for all $s\ge 3$.
\end{theorem}

Theorem~\ref{thm:odd_s} suggests a more delicate question: for $t = 2s+1$, can we increase the maximum possible value of $v$ from $s+2$ by incrementing the value of the parameter $N=(s+1)^2$ by $1$;
in other words, is $\scn(2s+1,(s+1)^2+1) > s+2$ for infinitely many $s$? 
Small examples appear hopeful: we demonstrate by example in Section~\ref{sec:cores} the new results that $\scn(7,17) \ge 6$ and $\scn(9,26) \ge 7$, corresponding to the cases $s=3$ and $s=4$. However, in Section~\ref{sec:ramsey} we show that such an increase is possible for only finitely many~$s$ (all of which can be shown to be at most~$14$). In fact, in Section \ref{sec:ramsey} we use Ramsey's theorem \cite{ramsey} (to our knowledge, a new tool in the study of suitable arrays) to show the surprising result that, for the parameters of both parts of Theorem~\ref{thm:colbourn}, the value of $N$ can be increased any fixed amount $\ell$ and yet $v$ can be increased from $s+2$ for only finitely many $s$:

\begin{theorem}
\label{thm:plusell}
\mbox{}
\begin{enumerate}

\item[(i)] 
For each nonnegative integer $\ell$, there exists $s_0$ (depending on $\ell$) such that
$\scn(2s,s(s+1)+\ell) = s+2$ for all $s \ge s_0$.

\item[(ii)]
For each nonnegative integer $\ell$, there exists $s_0$ (depending on $\ell$) such that
$\scn(2s+1,(s+1)^2+\ell) = s+2$ for all $s \ge s_0$.

\end{enumerate}
\end{theorem}

The remainder of the paper is structured in the following way.
In Section~\ref{sec:cores} we introduce suitable cores as equivalent objects to suitable arrays. In Section~\ref{sec:prelim} we establish some preliminary results to be used in later nonexistence proofs. In Section~\ref{sec:tt+1} we establish a link between suitable cores with parameters $t$ and $t+1$. In Sections~\ref{sec:oddcase} and~\ref{sec:ramsey} we prove the central results of the paper, Theorems~\ref{thm:odd_s} and~\ref{thm:plusell}, respectively.


\section{Suitable cores}
\label{sec:cores}

In this section, we recast the problem of finding suitable arrays as the equivalent problem of finding ``suitable cores'', as defined in \cite{colbourn} based on the proof of Theorem I of~\cite{dushnik}. 
As previously noted, the $N \times N$ array whose initial elements are $1,2,\dots,N$ is $(N,N,t)$-suitable for each $t \le N$ and so we may restrict attention to $(N,v,t)$-suitable arrays having $v \ge N$.
We begin with a straightforward lemma.

\begin{lemma}
\label{pushelt}
Let $A$ be an $(N,v,t)$-suitable array, and let $\al$ occur in the leftmost position of some row of $A$. Then the array $B$ obtained by moving an occurrence of $\al$ in a different row of $A$ to the rightmost position of its row results in another $(N,v,t)$-suitable array.
\end{lemma}

\begin{proof}
The symbol $\al$ precedes all other symbols in some row of $B$, and therefore precedes each set of $t-1$ other symbols in this row. Each symbol $\beta$ other than $\al$ precedes each set of $t-1$ other symbols in at least one row of $A$, and it still precedes the same $t-1$ symbols when some occurrence of $\al$ is moved rightwards to form~$B$.
\end{proof}

For example, we can transform the following $(5,7,3)$-suitable array on the left to the $(5,7,3)$-suitable array on the right by applying Lemma~\ref{pushelt} repeatedly (moving $\al=3$ in row 4, then $\al=4$ in row 4, then $\al=5$ in row 4, then $\al=6$ in row 5, then $\al=7$ in rows 1 through~4).

$$
\left[
\begin{array}{ccccccc}
3&1&2&7&4&5&6 \\
4&1&2&7&3&5&6 \\
5&7&2&1&3&4&6 \\
6&3&4&5&2&1&7 \\
6&7&2&1&3&4&5
\end{array}
\right]
\rightarrow
\left[
\begin{array}{c}
3 \\
4 \\
5 \\
6 \\
7
\end{array}
\right.
\framebox[0.8\width]{
$\begin{array}{cc}
1&2 \\
1&2 \\
2&1 \\
2&1 \\
2&1
\end{array}$
}
\left.
\begin{array}{cccc}
4&5&6&7 \\
3&5&6&7 \\
3&4&6&7 \\
3&4&5&7 \\
3&4&5&6
\end{array}
\right]
$$
Conversely, the boxed $5\times2$ subarray on the right can be transformed back into a $(5,7,3)$-suitable array by choosing 5 new symbols, prepending a different one to each row, and in each row appending the remaining 4 new symbols in arbitrary order.

In general, by applying Lemma~\ref{pushelt} repeatedly we can transform an $(N,v,t)$-suitable array $A$ 
into another $(N,v,t)$-suitable array $B$ having the following properties:

\begin{itemize}

\item
The leftmost column of $B$ consists of $N$ distinct symbols; call these the \ii{first symbols} of~$B$.

\item
Columns $2$ to $v-N+1$ of $B$ consist only of the $v-N$ symbols which are not first symbols.

\item
Columns $v-N+2$ to $v$ of $B$ consist only of first symbols.
\end{itemize}

We can also transform the $N \times (v-N)$ array formed from columns $2$ to $v-N+1$ of $B$ back into an $(N,v,t)$-suitable array by choosing $N$ new symbols, prepending a different one to each row, and in each row appending the remaining $N-1$ new symbols in arbitrary order. We call an $N \times (v-N)$ array that can be transformed into an $(N,v,t)$-suitable array by this procedure an \ii{$(N,v-N,t)$-suitable core}. In the example above, the boxed $5\times2$ subarray is a $(5,2,3)$-suitable core.

We see in this way that the existence of an $(N,v+N,t)$-suitable array is equivalent to the existence of an $(N,v,t)$-suitable core. Given $N$ and $t$, define $\scn(t,N)$ to be the largest $v$ for which an $(N,v,t)$-suitable core exists. This is consistent with the definition~\eqref{scn-defn}, and determining $\sun(t,N)$ is equivalent to determining~$\scn(t,N)$. 

We remark that Colbourn \cite{colbourn} established Theorem~\ref{thm:colbourn}~(i) by constructing an $(s(s+1),s+2,2s)$-suitable core for all $s\ge 2$, and Theorem~\ref{thm:colbourn}~(ii) by constructing an $((s+1)^2,s+2,2s+1)$-suitable core for all $s\ge 1$.

Figures~\ref{483sc}, \ref{955sc}, \ref{1767sc},~\ref{2679sc} show examples of suitable cores with parameters $(4,8,3)$, $(9,5,5)$, $(17,6,7)$, $(26,7,9)$, respectively. To our knowledge, suitable cores with these parameters (and their associated suitable arrays) were not previously known. The first two were found by hand, and the second two by interactive computer search \cite{jedwab-2048}.
These examples imply the bounds $\scn(3,4)\ge 8$, $\scn(5,9)\ge 5$, $\scn(7,17)\ge 6$, $\scn(9,26)\ge7$ mentioned in Section~\ref{sec:intro} as motivation for the explorations leading to Theorems~\ref{thm:odd_s} and~\ref{thm:plusell}.

We next give necessary and sufficient conditions for an array to be an $(N,v,t)$-suitable core; these are essentially contained in \cite[Lemma~1.1]{colbourn}. For an array $C$, symbol $\si$, and subset $T$ of symbols, denote by $\pre{C}(\si,T)$ the set of rows of $C$ for which $\si$ either starts a row or is preceded only by elements of~$T$. 

\begin{proposition}
\label{core}

Let $C$ be an $N\times v$ array. The following statements are equivalent:

\begin{enumerate}
\item[(i)]
$C$ is an $(N,v,t)$-suitable core.

\item[(ii)]
For each $s$ satisfying $0\le s\le t-1$, each symbol of $C$ precedes each subset of $s$ others in at least $t-s$ rows.

\item[(iii)]
For each symbol $\si$ of $C$ and for each subset $T$ of other symbols, $|\pre{C}(\si,T)| \ge t+1-v+|T|$.

\end{enumerate}
\end{proposition}

\begin{proof}

(i) $\iff$ (ii):

Construct an $N \times (v+N)$ array $A$ from $C$ by adding $N$ new symbols as first symbols and completing the rows of $A$ arbitrarily. From the discussion following Lemma~\ref{pushelt}, statement (i) is equivalent to the statement that $A$ is an $(N,v+N,t)$-suitable array. We now show that this is equivalent to statement~(ii).

Suppose that $C$ does not satisfy~(ii), so that for some $s$ satisfying $0\le s\le t-1$ there is a symbol $\si$ in $C$ and a set $S$ of $s$ other symbols in $C$ such that $\si$ precedes all elements of $S$ in at most $t-s-1$ rows of~$C$. Combine $S$ with the set of first symbols of the corresponding rows of $A$ to give a set of size at most $t-1$, and extend it if necessary to a set of size $t-1$. Then there is no row of $A$ in which $\si$ precedes all elements of this set, and so $A$ is not an $(N,v+N,t)$-suitable array.

On the other hand, suppose that $C$ satisfies~(ii). Let $\si$ be a symbol in $A$ and $S$ be a set of $t-1$ other symbols in~$A$. We shall show that $\si$ precedes all elements of $S$ in some row of~$A$. If $\si$ is a first symbol of $A$, this is immediate. Otherwise, let $S'$ be the set of elements of $S$ which are not first symbols. By assumption, $\si$ precedes all elements of $S'$ in at least $t-|S'|$ rows of~$A$. Since there are only $t-1-|S'|$ elements of $S$ which are first symbols, then at least one of these $t-|S'|$ rows of $A$ does not begin with an element of $S$; in that row, $\si$ precedes all of $S$ (by construction of $A$ from~$C$). Thus $A$ is an $(N,v+N,t)$-suitable array.

(ii) $\iff$ (iii):

Let $\si$ be a symbol of $C$, let $S$ be a set of $s$ other symbols of $C$ where $0 \le s \le t-1$, and let $T = [v] \setminus (S \cup \{\si\})$. Then $\si$ precedes all elements of $S$ in at least $t-s$ rows of $C$ if and only if $|\pre{C}(\si,T)|\ge t-s$.  Note that $|T| = v-s-1$, and that $0 \le s \le t-1$ is equivalent to the trivial conditions $v-t \le |T| \le v-1$.
\end{proof}

\begin{figure}[p]
\begin{minipage}{.5\textwidth}
	\centering
$
\begin{bmatrix}
2&1&4&3&6&5&8&7 \\
3&4&1&2&7&8&5&6 \\
5&6&7&8&1&2&3&4 \\
8&7&6&5&4&3&2&1
\end{bmatrix}
$
	\caption{A $(4,8,3)$-suitable core.}
	\label{483sc}
\end{minipage}%
\begin{minipage}{.5\textwidth}
	\centering
$
\begin{bmatrix}
1&2&3&5&4 \\
2&1&4&5&3 \\
3&1&4&5&2 \\
3&2&4&5&1 \\
4&1&3&5&2 \\
4&2&3&5&1 \\
5&1&4&2&3 \\
5&2&4&1&3 \\
5&3&4&1&2
\end{bmatrix}
$
	\caption{A $(9,5,5)$-suitable core.}
	\label{955sc}
\end{minipage}
\end{figure}

\begin{figure}[p]
\begin{minipage}{.5\textwidth}
	\centering
$
\begin{bmatrix}
1&4&*&*&*&* \\
1&2&5&*&*&* \\
1&6&5&*&*&* \\
2&4&*&*&*&* \\
2&3&*&*&*&* \\
2&6&5&*&*&* \\
3&5&*&*&*&* \\
3&2&1&*&*&* \\
3&6&1&*&*&* \\
4&1&3&*&*&* \\
4&5&3&*&*&* \\
4&6&*&*&*&* \\
5&1&3&*&*&* \\
5&4&2&*&*&* \\
5&6&*&*&*&* \\
6&2&1&*&*&* \\
6&3&4&*&*&*
\end{bmatrix}
$
	\caption{A $(17,6,7)$-suitable core (starred entries may be filled arbitrarily).}
	\label{1767sc}
\end{minipage}%
\begin{minipage}{.5\textwidth}
	\centering
$
\begin{bmatrix}
1&6&5&*&*&*&* \\
1&7&5&*&*&*&* \\
1&3&5&*&*&*&* \\
1&4&2&*&*&*&* \\
2&6&1&*&*&*&* \\
2&7&1&*&*&*&* \\
2&5&3&*&*&*&* \\
2&4&1&*&*&*&* \\
3&6&5&*&*&*&* \\
3&7&5&*&*&*&* \\
3&1&*&*&*&*&* \\
3&2&*&*&*&*&* \\
4&6&1&*&*&*&* \\
4&7&1&*&*&*&* \\
4&5&2&*&*&*&* \\
4&3&*&*&*&*&* \\
5&6&*&*&*&*&* \\
5&7&3&*&*&*&* \\
5&1&2&*&*&*&* \\
5&4&*&*&*&*&* \\
6&7&5&*&*&*&* \\
6&2&*&*&*&*&* \\
6&3&4&*&*&*&* \\
7&6&1&*&*&*&* \\
7&4&*&*&*&*&* \\
7&2&3&*&*&*&*
\end{bmatrix}
$
	\caption{A $(26,7,9)$-suitable core (starred entries may be filled arbitrarily).}
	\label{2679sc}
\end{minipage}
\end{figure}

We now briefly review some results due to Colbourn \cite{colbourn} on suitable cores, which allow the exact determination of $\scn(t,N)$ for all $N$ up to approximately $t^2/4$.
\begin{proposition}[{\cite[Section~1]{colbourn}}]
\label{smallscn}
\mbox{}
\begin{enumerate}
\item[(i)]
Suppose there exists an $(N,v,t)$-suitable core. Then $N\ge i(t+1-i)$ for $i = 1, 2, \dots, \min(v,t)$.

\item[(ii)]
Let $v\le (t+2)/2$. Then an $(N,v,t)$-suitable core exists if and only if $N\ge v(t+1-v)$.

\item[(iii)]
$\scn(t,N)=k$ for each $k \ge 0$ satisfying $k(t+1-k)\le N<(k+1)(t-k)$.

\end{enumerate}
\end{proposition}
\begin{proof}
\mbox{ }
\begin{enumerate}
\item[(i)]

Let $i$ be an integer satisfying $1\le i\le \min(v,t)$. Let $S$ be a subset of $[v]$ of size $i$ and let $\si \in S$. By Proposition~\ref{core}, $\si$ precedes all other elements of $S$ in at least $t+1-i$ rows. As $\si$ ranges over $S$ we obtain $i(t+1-i)$ mutually disjoint rows of the core. 

\item[(ii)]
Suppose an $(N,v,t)$-suitable core exists. Then we may take $i=v$ in (i) to obtain $N \ge v(t+1-v)$.

Now suppose that $N \ge v(t+1-v)$. Construct an $N \times v$ array $C$ whose first $v(t+1-v)$ rows have the following form: each symbol $i \in [v]$ starts a row $t+1-v$ times, and each symbol $j \in [v]$ that is distinct from symbol $i$ occurs directly after $i$ in at least one row. This is possible because $v\le (t+2)/2$ implies $t+1-v \ge v-1$.
We now use Proposition~\ref{core} to show that $C$ is an $(N,v,t)$-suitable core. Let $\si\in [v]$ and let $T$ be a set of symbols other than~$\si$. Then $\pre{C}(\si,T)$ contains the $t+1-v$ rows in which $\si$ appears first, as well as at least $|T|$ rows in which $\si$ appears directly after an element of $T$, for a total of at least $t+1-v+|T|$ rows.

\item[(iii)]
Let $k \ge 0$ satisfy $k(t+1-k) \le N < (k+1)(t-k)$. The range for $N$ given by these inequalities is nonempty exactly when $k \le (t-1)/2$. Therefore we may apply (ii) with $v=k$ and use the assumption $N \ge k(t+1-k)$ to show that there exists an $(N,k,t)$-suitable core and so $\scn(t,N)\ge k$. Then apply (ii) with $v=k+1$ and use the assumption $N < (k+1)(t-k)$ to show that there does not exist an $(N,k+1,t)$-suitable core and so $\scn(t,N)\le k$. We conclude that $\scn(t,N)=k$.

\end{enumerate}
\end{proof}

For fixed $t$, the smallest $N$ for which $\scn(t,N)$ is not determined by Proposition~\ref{smallscn}~(iii) is $N=(t/2)((t+2)/2)$ if $t$ is even, and $((t+1)/2)^2$ if $t$ is odd. These smallest undetermined cases can be written as $\scn(2s,s(s+1))$ when $t$ is even, and $\scn(2s+1,(s+1)^2)$ when $t$ is odd, and this motivates the constructions underlying Theorem~\ref{thm:colbourn}.


\section{Preliminary results}
\label{sec:prelim}

We shall use the following two lemmas in our nonexistence results for suitable cores.

\begin{lemma}
\label{reduction}
Removing all occurrences of a single symbol from an $(N,v,t)$-suitable core (and left-justifying the remaining symbols) results in an $(N,v-1,t)$-suitable core.
\end{lemma}

\begin{lemma}
\label{properties}
Suppose that $C$ is an $(N,v,t)$-suitable core.

\begin{enumerate}
\item[(i)]
Let $v \le t$. Then each $k \in [v]$ starts a row at least $t+1-v$ times.

\item[(ii)]
Let $v \le t+1$, let $k$ start a row exactly $t+1-v$ times, and let $j$ be another symbol. Then there is at least one row that starts with $j\,k$.

\item[(iii)]
Let $v \le t+2$, let $k$ start a row exactly $t+2-v$ times, and let $i,j$ be two other distinct symbols. If neither $i\,k$ nor $j\,k$ starts a row, then there is at least one row that starts with $i\,j\,k$ or $j\,i\,k$.

\end{enumerate}
\end{lemma}
\begin{proof}
\mbox{}

\begin{enumerate}
\item[(i)]
For $k \in [v]$, apply Proposition \ref{core} with $T=\emptyset$ to show that $|\pre{C}(k,\emptyset)|\ge t+1-v$. 

\item[(ii)]
Apply Proposition \ref{core} with $T=\set{j}$ to show that $|\pre{C}(k,\set{j})|\ge t+2-v$. Since $k$ starts a row exactly $t+1-v$ times, $k$ must be preceded by $j$ and by no other symbol in at least one row. 

\item[(iii)]
Apply Proposition \ref{core} with $T=\set{i,j}$ to show that $|\pre{C}(k,\set{i,j})|\ge t+3-v$. Since $k$ starts a row exactly $t+2-v$ times, $k$ must be preceded by one or both of $i$ and $j$, and by no other symbol, in at least one row. Excluding the cases $i\,k$ and $j\,k$ for the initial symbols of this row leaves only the cases $i\,j\,k$ and $j\,i\,k$.

\end{enumerate}
\end{proof}


\section{Suitable cores with parameters $t$ and $t+1$}
\label{sec:tt+1}

The following result links suitable cores with parameters $t$ and $t+1$.

\begin{theorem}
\label{thm:tt+1}
Suppose $\scn(t,N) \ge v$ and $N>v(t+1-v)$. Then $\scn(t+1,N+v-1) \ge v$.
\end{theorem}

\begin{proof}
We suppose that $C$ is an $(N,v,t)$-suitable core, where $N>v(t+1-v)$, and prove the result by constructing an $(N+v-1,v,t+1)$-suitable core~$D$.
Since $N>v(t+1-v)$, by the pigeonhole principle some symbol starts a row of $C$ at least $t+2-v$ times; relabel if necessary so that this symbol is~$v$. Form $D$ by adding $v-1$ rows to $C$, these rows starting with $1\,v,2\,v,\dots,(v-1)\,v$. We now show that $D$ is an $(N+v-1,v,t+1)$-suitable core using Proposition~\ref{core}. Let $\si\in [v]$ and let $T$ be a (possibly empty) set of symbols other than~$\si$. We distinguish two cases.

\begin{description}
\item[Case 1] $\si\ne v$. Then $\pre{C}(\si,T)$ consists of at least $t+1-v+|T|$ rows. Combine with the extra row of $D$ starting $\si\,v$ to give the necessary $t+2-v+|T|$ rows for $\pre{D}(\si,T)$.

\item[Case 2] $\si=v$. When $T$ is empty, the required condition $|\pre{D}(v,\emptyset)| \ge t+2-v$ is satisfied because of the $t+2-v$ rows that start with~$v$. When $T$ is nonempty, choose $c \in T$ and then the required condition $|\pre{D}(v,T)| \ge t+2-v+|T|$ is satisfied because there are at least $t+1-v+|T|$ rows in $\pre{C}(v,T)$ and the extra row of $D$ starting~$c\,v$.
\end{description}

\end{proof}

We note two important corollaries of Theorem~\ref{thm:tt+1}.

\begin{corollary}
\label{cor_extension}
Theorem~\ref{thm:colbourn}~(i) implies Theorem~\ref{thm:colbourn}~(ii) for $s>1$.
\end{corollary}
\begin{proof}
Apply Theorem~\ref{thm:tt+1} with $(N,v,t)=(s(s+1),s+2,2s)$ for each $s > 1$.
\end{proof}

\begin{corollary}
\label{splus3_parity}
Theorem~\ref{thm:plusell}~(ii) implies Theorem~\ref{thm:plusell}~(i).
\end{corollary}
\begin{proof}
We have the general result that $\scn(t,N+1) \ge \scn(t,N)$, by considering the addition of an arbitrary extra row to a suitable core. 
In view of Theorem~\ref{thm:colbourn}~(i), we then see that Theorem~\ref{thm:plusell}~(i) is equivalent to: given an integer~$\ell \ge 0$, we have $\scn(2s,s(s+1)+\ell) < s+3$ for all sufficiently large $s$. Likewise, in view of Theorem~\ref{thm:colbourn}~(ii), we see that Theorem~\ref{thm:plusell}~(ii) is equivalent to: given an integer~$\ell' \ge 1$, we have $\scn(2s+1,(s+1)^2+\ell') < s+3$ for all sufficiently large~$s$. 

Apply Theorem~\ref{thm:tt+1} with $(N,v,t)=(s(s+1)+\ell,s+3,2s)$ to show that if $\scn(2s,s(s+1)+\ell) \ge s+3$ then $\scn(2s+1,(s+1)^2+\ell+1) \ge s+3$, and take $\ell' = \ell+1$.
\end{proof}


\section{Proof of Theorem~\ref{thm:odd_s}}
\label{sec:oddcase}

In this section we prove Theorem~\ref{thm:odd_s}. As described in Section~\ref{sec:intro}, the nonexistence result of Theorem~\ref{thm:odd_s} holds for all $s \ge 3$ but not for $s=1$ and $s=2$ (see Figures~\ref{483sc} and~\ref{955sc}). We shall see where the condition $s \ge 3$ is required in the proof of Theorem~\ref{thm:odd_s}, and why the proof does not apply to the $(9,5,5)$-suitable core shown in Figure~\ref{955sc}.

\begin{proof}[Proof of Theorem~\ref{thm:odd_s}]
In view of Theorem~\ref{thm:colbourn}~(ii), it is required to prove that for each $s \ge 3$ there does not exist an $((s+1)^2,s+3,2s+1)$-suitable core.
Suppose, for a contradiction, that $C$ is such a suitable core.

Note from Lemma \ref{properties}~(i) that each of the $s+1$ symbols of an $((s+1)^2,s+1,2s+1)$-suitable core $C'$ starts a row at least $s+1$ times, and since this accounts for all $(s+1)^2$ rows of $C'$ we have that
\begin{equation}
\label{equidistribution}
\mbox{each symbol of an $((s+1)^2,s+1,2s+1)$-suitable core starts a row exactly $s+1$ times}.
\end{equation}
It follows that
\begin{align}
& \mbox{no symbol $\si$ of $C$ starts a row more than $s+1$ times after removing all occurrences of} \nonumber \\[-1ex]
& \mbox{zero, one or two other symbols from $C$}, \label{nots+1}
\end{align}
for otherwise we could remove all occurrences of another two, one or zero other symbols, respectively, and by Lemma~\ref{reduction} would obtain an $((s+1)^2,s+1,2s+1)$-suitable core in which $\si$ starts a row more than $s+1$ times, contrary to~\eqref{equidistribution}.

Relabel if necessary so that the symbols of $C$ are elements of $[s+3]$ and the number of rows starting with $i$ is nondecreasing with~$i$. By Lemma~\ref{properties}~(i), each of the $s+3$ symbols of $C$ starts a row at least $s-1$ times. This accounts for $(s+3)(s-1) = (s+1)^2-4$ of the $(s+1)^2$ rows of $C$, leaving four more rows to account for. By \eqref{nots+1} (with ``zero''), there are three possible distributions for the symbols that start these four rows:

\begin{description}

\item[Case 1] Symbols $1$ to $s-1$ each start a row exactly $s-1$ times, and symbols $s$ to $s+3$ each start a row exactly $s$ times.

By Lemma \ref{properties}~(ii), each symbol that starts a row in $C$ exactly $s-1$ times must appear second after each other symbol. Therefore $C$ contains a row starting $j\,k$ for each $k=1,2,\dots,s-1$ and for each $j \ne k$. The $s+3$ other rows of $C$ each start with a different symbol of $[s+3]$. 

Among these $s+3$ rows, no $k$ from $1$ to $s-1$ can appear second, otherwise $C$ would contain two rows starting $j\,k$ for some $j$ and a row starting $i\,k$ for some $i$ distinct from $j$ for which $i \ge s$; removing all occurrences of symbols $i$ and $j$ from $C$ would then leave at least $(s-1)+3 = s+2$ rows starting with $k$, contradicting~\eqref{nots+1}.

Furthermore, among these $s+3$ rows, no $k$ from $s$ to $s+3$ can appear second more than once, otherwise $C$ would contain a row starting $i\,k$ and a row starting $j\,k$ for $i,j$ not necessarily distinct; removing all occurrences of symbols $i$ and $j$ from $C$ would again leave at least $s+2$ rows starting with $k$, contradicting~\eqref{nots+1}.

Therefore each of the $s+3$ rows must contain a distinct symbol from $s$ to $s+3$ in its second position, which gives the contradiction $s+3 \le 4$.

\item[Case 2] Symbols $1$ to $s$ each start a row exactly $s-1$ times, symbols $s+1$ and $s+2$ each start a row exactly $s$ times, and symbol $s+3$ starts a row exactly $s+1$ times.

By Lemma \ref{properties}~(ii), $C$ contains a row starting $i\,j$ for each $j=1,2,\dots,s$ and for each $i \ne j$. There is only one other row of $C$ and it starts with~$s+3$.

By Lemma \ref{properties}~(iii), since $s+1$ and $s+2$ each start a row in $C$ exactly $s$ times, for each $i,j$ satisfying $1 \le i < j \le s$ either $C$ contains rows starting $i\,j\,(s+1)$ and $j\,i\,(s+2)$ or $C$ contains rows starting $i\,j\,(s+2)$ and $j\,i\,(s+1)$. 

It follows that $s+3$ never occurs second or third in a row of~$C$ that starts with 1, 2, or 3, and, because $s \ge 3$, no row of $C$ starts with the symbols $1,2,3$ in any order.

But by Proposition~\ref{core}, $\pre{C}(s+3, \set{1,2,3})$ contains at least $s+2$ rows. Since there are exactly $s+1$ rows of $C$ starting with $s+3$, there is some row of $C$ that does not start with $s+3$ in which $s+3$ is preceded only by elements of $\set{1,2,3}$. This gives the required contradiction.

\item[Case 3] Symbols $1$ to $s+1$ each start a row exactly $s-1$ times, and symbols $s+2$ and $s+3$ each start a row exactly $s+1$ times.

By Lemma~\ref{properties}~(ii) with $j=1$ and $k=2,3,\dots,s+1$, there are at least $s$ rows starting with~$1$. This contradicts that symbol $1$ starts a row exactly $s-1$ times.

\end{description}

\end{proof}

Note that the proof of Theorem~\ref{thm:odd_s} does not apply to the $(9,5,5)$-suitable core shown in Figure~\ref{955sc}, because its first row starts with the symbols $1, 2, s+1$ where $s=2$.

Colbourn \cite{colbourn} states without proof that $\scn(2s,s(s+1))= s+2$ for all $s \ge 2$, in other words that the inequality of Theorem~\ref{thm:colbourn}~(i) is actually an equality. This result can be recovered using similar techniques to those in the above proof of Theorem~\ref{thm:odd_s}.


\section{Proof of Theorem~\ref{thm:plusell}}
\label{sec:ramsey}

In this section we prove Theorem~\ref{thm:plusell}. We first establish two auxiliary lemmas.

Lemma~\ref{counting1} shows that if, for a set $A$, we associate each element of $A$ with a subset of $A$ of size at most $m$, then some element of $A$ appears in at most $m$ of the subsets.

\begin{lemma}
\label{counting1}
Let $d$ and $m$ be positive integers. Let $A$ be a set of size $d$ and let $g$ be a function from $A$ to subsets of $A$ of size at most~$m$. Then there exists $k\in A$ for which $\set{\ell\in A:k\in g(\ell)}$ has at most $m$ elements.
\end{lemma}
\begin{proof}
Let $f(k)=|\set{\ell\in A:k\in g(\ell)}|$. Then $\sum_{k\in A} f(k)\le md$, and so the mean of $f(k)$ over $k\in A$ is at most $m$. So $f(k)\le m$ for some $k\in A$.
\end{proof}

We next refine Lemma~\ref{counting1} to show that, if $d$ is large enough, we can choose $e$ elements of $A$, each of which appears in none of the subsets associated with the other $e-1$ elements.

\begin{lemma}
\label{counting2}
Let $e$ and $m$ be positive integers and $d\ge (e-1)(2m+1)+1$. Let $A$ be a set of size $d$ and let $g$ be a function from $A$ to subsets of $A$ of size at most $m$. Then there exists a subset $B$ of $A$ of size $e$ such that $j\notin g(i)$ for all distinct $i,j\in B$.
\end{lemma}
\begin{proof}
The proof is by induction on $e\ge 1$, with $m$ and $d$ satisfying the stated conditions.

If $e=1$, then $|A| = d\ge 1$. Then simply choose $B$ to comprise one element in~$A$.

Now let $e>1$ and assume the statement is true for all positive integers less than $e$ and for all $m$ and $d$ satisfying the stated conditions. By Lemma \ref{counting1}, there exists $k\in A$ such that $|\set{\ell\in A:k\in g(\ell)}|\le m$. Let $S=\set{k}\cup g(k)\cup\set{\ell\in A:k\in g(\ell)}$ and $A'=A\setminus S$. Note that $|S|\le 2m+1$, and so $|A'|\ge d-(2m+1)\ge (e-2)(2m+1)+1$. We then define a function $g'$ from $A'$ to subsets of $A'$ of size at most $m$ as follows: for each $b\in A'$, $g'(b)=g(b)\cap A'$. Clearly $|g'(b)|\le m$ for each $b\in A'$.

By the inductive hypothesis applied to $A'$ and $g'$, there exists a subset $B'$ of $A'$ of size $e-1$ such that $j\notin g'(i)$ for all distinct $i,j\in B'$. Then we let $B=B'\cup \set{k}$. Since $k\notin B'$, we have $|B|=e$. Now let $j\in B$. We complete the induction by showing that $j \notin g(i)$ for all $i \in B\setminus \set{j}$.

\begin{description}
\item[Case 1] $j\ne k$. Then $j\notin g(k)$ by definition of $S$ and $A'$, since $j\in A'$ and $g(k)\subseteq S$. Also, $j \notin g(i)$ for all $i\in B'\setminus \set{j}$, since $j\notin g'(i)$ and $g'(i)=g(i)\cap A'$. Together this gives $j \notin g(i)$ for all $i\in B\setminus \set{j}$.

\item[Case 2] $j=k$. Then $k\notin g(i)$ for all $i\in B\setminus \set{k}$, since $\set{\ell\in A:k\in g(\ell)}\subseteq S$ and $i\in A'$.
\end{description}
\end{proof}

We are now ready to prove Theorem~\ref{thm:plusell}.

\begin{proof}[Proof of Theorem~\ref{thm:plusell}]
By Corollary~\ref{splus3_parity}, it is sufficient to prove only part~(ii) of the theorem. In view of Theorem~\ref{thm:colbourn}~(ii), it is required to prove that for all sufficiently large $s$ there does not exist an $((s+1)^2+\ell,s+3,2s+1)$-suitable core (where $\ell$ is a fixed nonnegative integer). Suppose, for a contradiction, that there is some arbitrarily large $s$ for which $C$ is such a suitable core.

Relabel if necessary so that the symbols of $C$ are elements of $[s+3]$ and the number of rows starting with $i$ is nondecreasing with~$i$. By Lemma~\ref{properties}~(i), each of the $s+3$ symbols of $C$ starts a row at least $s-1$ times. This accounts for $(s+3)(s-1) = (s+1)^2-4$ of the $(s+1)^2+\ell$ rows of $C$, leaving $\ell+4$ more rows to account for. The number of symbols that start a row more than $s-1$ times is then at most $\ell+4$. Let $c$ be the number of symbols that start a row exactly $s-1$ times, so that each of $1$ to $c$ starts a row exactly $s-1$ times and $c \ge s+3-(\ell+4) = s-\ell-1$.

By Lemma~\ref{properties}~(ii), $C$ contains a row starting $i\,j$ for each $j=1,2,\dots,c$ and for each $i\ne j$. Form $C'$ from $C$ by deleting the first such row for every such pair $(i,j)$. Then in $C'$, each of $1$ to $c$ starts a row exactly $m:=s-c$ times. Since $c\ge s-\ell-1$, we have $m\le \ell+1$.

The number of elements in $\set{c+1,\dots, s+3}$ is $s-c+3=m+3$. For each $i \in [c]$, since there are $m$ rows of $C'$ starting with $i$ there are at least $3$ elements of $\set{c+1,\dots, s+3}$ which do not appear second after $i$ in $C'$, and so do not appear second after $i$ in~$C$.
We may therefore define a function $f$ from $[c]$ to 3-subsets of $\set{c+1,\dots, s+3}$, such that $f(i)=\set{j_1,j_2,j_3}$ where $j_1,j_2,j_3$ do not appear second after $i \in [c]$.

Now choose $s$ to be large enough to force $c\ge\binom{m+3}{3}(d-1)+1$ (via the inequality $c\ge s-\ell-1$), where $d\ge 1$ is an integer to be determined later. Then, by the pigeonhole principle, there exists a set of $d$ numbers $A=\set{a_1, a_2, \dots, a_d}$ in $[c]$ for which $f(a_1)=f(a_2)=\dots=f(a_d)$. Let $\set{k_1,k_2,k_3}=f(a_1)$.

Next choose $d\ge (e-1)(2m+1)+1$, where $e\ge 1$ is an integer to be determined later. Define the function $g$ from $A$ to subsets of $A$ via: for each $a\in A$, $g(a)$ is the set of elements of $A$ appearing second in the rows of $C'$ that start with $a$; so $g(a)$ has size at most~$m$. By Lemma \ref{counting2}, there exists a subset $B=\set{b_1,\dots, b_e}$ of $A$ of size $e$ such that $b_i\notin g(b_j)$ for all distinct $i,j$.
It follows that no row of $C'$ starting with an element of $B$ has an element of $B$ appearing second. By the construction of $C'$ from $C$, we conclude that for each pair of distinct elements $b_x,b_y$ of $B$ there is exactly one row of $C$ starting~$b_x\,b_y$.

Now associate with $C$ a graph $G$ whose vertex set is~$[e]$.  For each $x,y\in [e]$, there is at least one element of the set $\set{k_1,k_2,k_3}$ that precedes the other two in neither the row starting $b_x\,b_y$ nor the row starting $b_y\,b_x$; choose one such element and color the edge between vertices $x,y$ with color $1$ if the choice is $k_1$, color $2$ if it is $k_2$, and color $3$ if it is $k_3$. The resulting graph $G$ is a complete graph $K_e$ on $e$ vertices whose edges are colored from a set of $3$ colors.

Recall that $k_1,k_2,k_3$ are in $\set{c+1,\dots, s+3}$ by definition of $f$. Now in $C$, the symbols $k_1, k_2, k_3$ start a row $s-1+r_1,s-1+r_2,s-1+r_3$ times, respectively, for some positive integers $r_1,r_2,r_3$. Let $T$ be a subset of $B$ of size $r_1+1$. Then by Proposition \ref{core}, $\pre{C}(k_1,T)$ contains at least $s+r_1$ rows. Then there is some row of $C$, that does not start with $k_1$, in which $k_1$ is preceded only by elements of~$T$. By the definition of $f$, $k_1$ does not appear second after $b_1, b_2, \dots, b_e$, so this row starts $b_x\,b_y$ for some distinct elements $b_x,b_y$ of $T$, and certainly in this row $k_1$ precedes $k_2$ and $k_3$. The edge joining vertices $x$ and $y$ of $G$ is therefore not~colored~$1$. Since this applies over all subsets $T$ of $B$ of size $r_1+1$, this means that $G$ does not contain a $K_{r_1+1}$ of color $1$.
A similar analysis holds for $k_2$ ($T$ has size $r_2+1$) and $k_3$ ($T$ has size $r_3+1$), and so $G$ also contains neither a $K_{r_2+1}$ of color $2$ nor a $K_{r_3+1}$ of color $3$.

However, by Ramsey's theorem \cite{ramsey}, for some $v\ge 1$, denoted $R(r_1+1,r_2+1,r_3+1)$, each edge coloring of a complete graph on $v$ vertices using three colors contains either a $K_{r_1+1}$ of color $1$, or a $K_{r_2+1}$ of color $2$, or a $K_{r_3+1}$ of color $3$. Choose $e=v$ to give the required contradiction.

\end{proof}


\section{Open problems}
We conclude with some open problems suggested by the results of this paper.  

\begin{enumerate}

\item
Theorem~\ref{thm:plusell}~(ii) specifies the existence of $s_0$ for which an expression involving $\scn$ holds for all $s \ge s_0$; but our proof, using Ramsey's theorem, does not determine a minimum~$s_0$. Given a nonnegative integer $\ell$, what is the smallest possible value of $s_0$ and how does it grow with~$\ell$? 

\item
The examples of suitable cores given in Figures~\ref{1767sc} and~\ref{2679sc} show that the inequality
\begin{equation}
\scn(2s+1,(s+1)^2+\ell) > s+2
\label{ineq}
\end{equation} 
holds for $\ell=1$ and $s=3,4$.
However, Theorem~\ref{thm:plusell}~(ii) shows that \eqref{ineq} holds for only finitely many $s$ when $\ell$ is a fixed positive integer. 
But if $\ell$ is allowed to increase with $s$ then \eqref{ineq} can hold for infinitely many~$s$: substitute $s+1$ for $s$ in Theorem~\ref{thm:colbourn}~(ii), and use the general result from Proposition~\ref{core} that $\scn(t,N)\ge v$ implies $\scn(t-1,N) \ge v$, to show that $\ell = 2s+3$ suffices. Does a function of $s$ growing more slowly than $2s+3$ suffice for \eqref{ineq} to hold for infintely many~$s$? Does a function of $s$ growing more slowly than linearly with~$s$ suffice?

\end{enumerate}


\section*{Acknowledgements}

The authors thank the organizers of the \emph{22nd Coast Combinatorics Conference}, Kailua-Kona, HI, February 2015, where they first learned of the problem of constructing suitable sets of permutations.


\end{document}